\documentclass[11pt]{article}

\usepackage{amssymb}
 \usepackage{amsthm}
\usepackage{amsmath}

 \usepackage{hyperref}

\newtheorem{theorem}{Theorem}[section]
\newtheorem{lemma}[theorem]{Lemma}
\newtheorem{proposition}[theorem]{Proposition}

\newcommand{\N}{\mathbb N}
\newcommand{\F}{\mathbb F}
\newcommand{\Q}{\mathbb Q}
\newcommand{\Z}{\mathbb Z}
\newcommand{\sti}[2]{\left[ \begin{array}{c}{#1} \\ {#2}\end{array} \right]}
\newcommand{\stiny}[2]{{\tiny \left[ \begin{array}{c}{#1} \\ {#2}\end{array} \right]}}

\usepackage{color}

\begin{document}

	\title{On the number of monochromatic solutions of integer linear systems on Abelian groups}


	\author{Oriol Serra\thanks{Departament de Matem\`atica Aplicada 4,
	           Universitat Polit\`ecnica de Catalunya,
	           Barcelona, Spain.
	            Email: \texttt {oserra@ma4.upc.edu}}
	        \and
            Llu\'{i}s Vena\thanks{Department of Mathematics,
			  University of Toronto,
			    Toronto, Canada.
			 Email: \texttt{lluis.vena@utoronto.ca}}
			}
	\date{\today}
\maketitle

	\begin{abstract}
	Let $G$ be a finite abelian group with exponent $n$, and let $r$ be a positive integer. Let $A$ be a $k\times m$ matrix with integer entries. We show that if $A$ satisfies some natural conditions and $|G|$ is large enough then, for each   $r$--coloring of $G\setminus \{ 0\}$, there is $\delta$ depending only on $r,n$ and  $m$ such that the homogeneous linear system $Ax=0$ has at least $\delta |G|^{m-k}$ monochromatic solutions. Density versions of this counting result are also addressed.
	\end{abstract}


	\section{Introduction} \label{sec.intro}

	A central topic in Arithmetic Ramsey Theory is the study of monochromatic solutions of homogeneous linear systems in a colouring of the ambient group. Let $A$ be a $k\times m$ matrix with integer coefficients. The linear system $Ax=0$ is said to be  {\it partition regular} if every finite colouring of the nonzero integers  has a monochromatic solution of the system. A celebrated theorem by Rado~\cite{rad33}  characterizes such partition regular systems in the integers in terms of the so--called columns condition over $\Q$.

	More generally, we say that a $(k\times m)$ matrix $A$ with coefficients in a ring $R$ satisfies the
	\emph{columns condition} over  $R$ if we can order the column vectors $A^1, \ldots,
	A^m$ and find $1\leq k_1 < k_2 \cdots < k_t = m$ (with $k_0=0$) such that, if we
	set
	\begin{displaymath} 
	S^i=\sum_{j=k_{i-1}+1}^{k_i} A^j, 
	\end{displaymath} 
	we have that

	\begin{description} 
	\item[(i)] $S^1=0$ in $R^k$. 
	\item[(ii)] for $1<i\leq t$,
	$S^i$ can be expressed as a linear combination of $A^1,\ldots, A^{k_{i-1}}$
	with coefficients in $R$.
	\end{description}

	Deuber~\cite{deu75} extended Rado's characterization  to general abelian groups.
	Since the definition of partition regular systems involves all finite colourings, the statement holds only in infinite groups. A finitistic version of the Ramsey statement was given by Bergelson, Deuber and Hindman \cite{berdeuhind92} for vector spaces over finite fields.

	\begin{theorem}[Bergelson, Deuber and Hindman \cite{berdeuhind92}] \label{t.ff_1_solution}
	Let $F$
	be a finite field and $A$ a $(k\times m)$ matrix with coefficients
	in $F$. The following statements are equivalent.

	{\rm (i)} For every $r\in \N$ there is $n(r,|F|,m)\in N$ such that every
	$r$--coloring of $F^n\setminus \{ 0\}$ with $n\ge n(r,|F|,m)$ has a
	monochromatic solution of the homogeneous linear system $Ax=0$.

	{\rm (ii)} The matrix $A$ verifies the columns condition over $F$.
	\end{theorem}

	Counting versions of the above existence results start with Varnavides~\cite{var59} for  the theorem of van der Waerden, by showing that the number of $3$-term monochromatic arithmetic progressions in finite colorings of the integers is a positive fraction of its total number. Frankl, Graham and R\"odl \cite{frangrahrod88} extended the above result to partition regular linear systems in the integers,  by proving that the set of monochromatic solutions has positive density in the set of all solutions.

	\begin{theorem}[Frankl, Graham, R\"odl \cite{frangrahrod88}]
	Let $r$ be a positive integer. Assume the $k\times m$ integer matrix $A$ satisfies the column condition over $\Q$. Then, there exists a constant $c=c(r,A)>0$ such that in every $r$-colouring of the integer interval $[1,N]$ there are at least $c N^{m-k}$ monochromatic solutions to the linear equation $Ax=0$.
	\end{theorem}

	In the same paper the authors  also characterize the linear systems which are {\it density regular}, namely, the ones which have solutions in any set of integers with positive upper density. They also show that the number of solutions of a density regular system in a set with positive upper density is a positive fraction of the total number of solutions (in this case the constant of proportionality depends on the density of the considered set).

	The result in \cite{frangrahrod88} for partition regular systems relies on van der Waerden's theorem on the existence of monochromatic  arithmetic progressions in any finite coloring of a  sufficiently long integer interval. The result for density regular systems uses Szemer\'edi's theorem \cite{szem75} on  the existence of arbitrarily long arithmetic progressions in sets of integers with positive upper density and the multidimensional Szemer\'edi theorem proved by Furstenberg and Katznelson \cite{furkatz78}. These three results cannot be applied in abelian groups with finite torsion.

	Recall that the exponent of a group $G$ is the least integer $n$ for which $g^n=1$, for all $g\in G$. In this work we show that, under certain hypothesis similar to the columns condition imposed on the  matrix $A$, the number of monochromatic solutions to $Ax=0$ in an abelian group $G$ with exponent $n$  is at least $\delta |G|^{m-k}$ for some $\delta$ depending only on $r, n$ and $m$ and sufficiently large $|G|$. This applies to finite fields (Theorem~\ref{t.num_sol_ff}, where the coefficients of the matrix can be taken in the field.)  More precisely, let  $A$ be a $(k\times m)$ matrix with integer entries. We say that $A$ satisfies the \emph{$n$-columns condition} if it satisfies the $\Z_n$-columns condition when its entries are considered to be in $\Z_n$ (all the operations and coefficients are thought to be in $\Z_n$.) Our main result is the following one:

	\begin{theorem}[Number of solutions for bounded torsion abelian groups] \label{t.num_sol_bound_tor}
	 Let $G$ be a finite abelian group  with exponent $n$. Let $r$ be a positive integer and let $A$ be a $(k\times m)$ matrix with integer entries.

	Assume that $A$ satisfies the $n$--columns condition. There is a constant $c>0$ depending only on $r, n$ and $m$ such that, if $|G|$ is sufficiently large, every $r$--colouring of $G\setminus \{ 0\}$ has at least $ c|G|^{m-k}$ monochromatic solutions of the equation $Ax=0$ in  $G\setminus \{0\}$.
	\end{theorem}

	The general idea of the proof is as follows. Assume that the number of solutions is asymptotically smaller than a constant fraction of $|G|^{m-k}$. We use a Ramsey result to ensure a monochromatic solution within a popular substructure. We then   use an appropriate version of the Removal Lemma to observe that, if the number of solutions is small, we can remove all the solutions  by deleting few elements. However, if the number of removed elements is small, then one of the popular structures should survive intact   still containing a monochromatic solution and we reach a contradiction.

	The argument is first illustrated for the simpler case of finite fields in Section~\ref{sec.fin_field_all}. The Ramsey result in this case is  \cite[Theorem~2.4]{berdeuhind92} and  the Removal Lemma for linear systems in finite fields is supplied by Shapira~\cite{sha10} and by Kr\'al' and the authors~\cite{kraserven11}. 

	The scheme of the proof of the main result , Theorem
\ref{t.num_sol_bound_tor},  is analogous to the one for the case of finite
fields.
Unfortunately, direct application of the above case does not give enough
solutions for general finite abelian groups with bounded
exponent and we need some preliminary results.

We first prove, in Section \ref{sec:ramsey},  a specific  Ramsey result we
need,  Lemma~\ref{l.cycl_1_solution},
which guarantees  a particular monochromatic structure  containing
solutions of  linear systems. The proof is
an adaptation of the one by  Spencer~\cite{spen79} of the Vector Space 
Ramsey Theorem  by Graham, Leeb and Rothschild~\cite{grahleebrot72}.
In the last section we briefly discuss the reason we include a complete
proof of   Lemma~\ref{l.cycl_1_solution}.

	In Section~\ref{sec:counting} we give an asymptotic counting result for the number of subgroups of a given type in an abelian group $G$ by using a result by Yeh~\cite{yeh48} on the number of  subgroups of a $p$--group. With these two ingredients and the version of the Removal Lemma for linear systems in abelian groups (Lemma~\ref{l.rm_eqsys_ab})  we proceed to the proof of Theorem~\ref{t.num_sol_bound_tor} in Section~\ref{sec.cyclic_groups_all}.  

	 Finally, Section~\ref{s.dens} discusses the density case. We give a characterization, Theorem~\ref{t.num_sol_dens}, of the matrices $A$ with integer entries  which are density regular for every finite abelian group. The paper closes with some final remarks.


	\section{Number of monochromatic solutions inside $\F_q^N$}
	\label{sec.fin_field_all}


	For this section, we let $A$ be a $k\times m$ matrix with coefficients in a finite field $F$ of order $q=p^l$.
	Let $\chi:F^N\to [r]$ be a colouring with $r$ colours. We are interested in
	solutions of the system $Ax=0$, with $x=(x_1,\ldots,x_m)$, $x_i\in
	F^N\setminus \{0\}$ and $\chi(x_1)=\chi(x_2)=\cdots=\chi(x_m)$ (monochromatic solutions.)

	Bergelson, Deuber and Hindman~\cite{berdeuhind92} gave the characterization of    partition regular systems for finite
	fields stated in Theorem~\ref{t.ff_1_solution}.  In this section we prove the following counting version of the result:

	\begin{theorem}[Number of monochromatic solutions in Finite Fields]
	\label{t.num_sol_ff}
	Let $F$ be a finite field with $q=p^l$ elements, let $k$,
	$m$, $N$, $r$ be positive integers, $m\geq k$, and let $A$ be a $k\times m$ matrix with
	coefficients in $F$. Assume that $A$ satisfies the $F$-columns condition. Then,
	for any coloring of the elements of $F^N$ with $r$ colors and sufficiently large $N$, there exists a constant
	$c=c(r,q,m)>0$ such that the system $Ax=0$ with $x=(x_1,\ldots,x_m)$ and
	$x_i\in F^N$ has at least $ c q^{N(m-k)}$ monochromatic solutions.
	\end{theorem}

	The proof of Theorem~\ref{t.num_sol_ff} is a simple combination of the Ramsey result Theorem~\ref{t.ff_1_solution} and the  Removal Lemma for finite fields, Lemma~\ref{t.rm_syseq_ff} below, proved independently by
	Shapira \cite{sha10} and Kr\'al' and the authors \cite{kraserven11}. It illustrates the proof strategy described in the Introduction when there are no additional difficulties. 

	\begin{lemma}[Removal Lemma for systems of equations \cite{sha10}, \cite{kraserven11}] \label{t.rm_syseq_ff}
	For all positive integers $k$ and $m$, $k\le m$, and every $\epsilon>0$, there exists $\delta=\delta(\epsilon,m)>0$ such that the following holds: Let $F = \mathbb{F}_{q}$ be the finite field of order $q$ and $X_1,\ldots,X_m$ be subsets of $F$, let $A$ be a $(k\times m)$ matrix with coefficients in $F$.

	If there are at most $\delta q^{m-k}$ solutions of the system $Ax=0$, $x=(x_1,\ldots,x_m)$, with $x_i\in X_i$, then there exist sets $X_1',\ldots,X_m'$ with $X'_i\subseteq X_i$ and $|X_i\setminus X_i'| \le \epsilon q$ such that there is no solution of the system $Ax=0$ with $x_i\in X_i'$.
	\end{lemma}

	\begin{proof}[Proof of Theorem~\ref{t.num_sol_ff}]

	Let $F=\F_q$ be the finite field over $q$ elements. Let $F^N$ be an $N$-dimensional space over $F$. Let $r$ be the number of colours and let $m$ be the number of columns of the matrix $A$. Denote by $Y_i$   the set of elements  coloured $i$, $i=1,\ldots ,r$. Let $M\ge n(r,q,m)$ be such that $F^M$ contains a monochromatic solution according to  Theorem~\ref{t.ff_1_solution}. 

	At this point, we apply $r$ times Lemma~\ref{t.rm_syseq_ff}, one for each
	colour, with $\epsilon$ to be specified later and $X_1=X_2=\cdots=X_m=Y_i$.
	By Lemma~\ref{t.rm_syseq_ff} there is  $\delta=\delta (\epsilon ,m)$ such that,   if the number of monochromatic solutions is at most $\delta q^{N(m-k)}$, then we obtain sets
	  $Y_i'\subset Y_i$ with $|Y_i'|\leq m \epsilon q^N$ such that
	$S=F^N\setminus \bigcup_{i=1}^r Y_i'$ has no monochromatic solution of our linear system.

	The number of $M$--dimensional subspaces in $F^N$ is given by the Gaussian coefficient ${N \choose M}_q $. 
	As we have removed at most $rm \epsilon q^N$ elements and each of them belongs to at most 
	\begin{displaymath}
		{ N \choose M}_q \frac{q^M-1}{q^N-1} = {N-1 \choose M-1}_q,
	\end{displaymath}
	$M$--dimensional subspaces, we have removed at most 
	$$
	rm \epsilon q^N {N-1 \choose M-1}_q
	$$
	such spaces. We observe that 
	$$
	\frac{{N\choose M}_q}{rmq^N{N-1\choose M-1}_q}=\frac{q^N-1}{rmq^N(q^M-1)}\to \frac{1}{rm (q^M-1)} \; \mbox{ as } N\to \infty.
	$$
	Hence we can choose $\epsilon$ independently of $N$, for sufficiently large $N$, such that
	$$
	0<\epsilon<\frac{{N\choose M}_q}{rmq^N{N-1\choose M-1}_q},
	$$ 
	so that there  is at least one $M$--subspace in $S$. By Theorem~\ref{t.ff_1_solution}, $S$ still contains a  monochromatic solution. This contradicts Lemma~\ref{t.rm_syseq_ff}, so that the system has at least $\delta q^{N(m-k)}$ monochromatic solutions, completing the proof of the Theorem. 
	   \end{proof}


	\section{A Ramsey result}\label{sec:ramsey}

	The scheme of the proof of Theorem \ref{t.num_sol_bound_tor} is the same as the one in Section~\ref{sec.fin_field_all} for the case of finite fields. Unfortunately, direct application of  Theorem~\ref{t.num_sol_ff} when $n$ is not a prime does not give us enough solutions, as the number of subgroups isomorphic to $\Z_p^N$ inside $G$, for any prime $p$ dividing $n$, is not large enough. In order to overcome this difficulty, we prove a specific Ramsey result, Lemma~\ref{l.cycl_1_solution} below. This Section is devoted to the proof of this result.

	\begin{lemma}[Solutions outside the finite fields] \label{l.cycl_1_solution}
	Let $n$ be a composite positive integer.
	Let $A$ be a $k\times m$ integer matrix satisfying the $n$-columns condition.
	There exists an $M=M(r,n,m)$ such that,
	for any coloring of the elements of $\mathbb{Z}_n^M\setminus \{0\}$
	with $r$ colors, there exists a monochromatic solution $x=(x_1,\ldots ,x_m)$ to the system $Ax=0$,
	with $x_i\in \mathbb{Z}_n^M \setminus \bigcup_{p | n} \mathbb{Z}_p^M$.
	Moreover, the order of all the $x_i$'s is $n$.
	\end{lemma}

		The proof of Lemma~\ref{l.cycl_1_solution}  follows the ideas  in Bergelson, Deuber, Hindman \cite{berdeuhind92}, which come back to Deuber \cite{deu75}, for the proof of Theorem \ref{t.ff_1_solution}. For the case of finite fields, this approach uses the Vector Space Ramsey Theorem of Graham, Leeb and Rothschild~\cite{grahleebrot72}. In our context we also need the following version of this result, Lemma \ref{lem:mon_subg} below. Before stating the Lemma let us fix some terminology.   

	We denote the standard generating set of $\Z_n^m$ by $\{e_1,\ldots ,e_m\}$, where $e_i$ has all entries zero but the $i$--th entry equal to one. Let $H<\Z_n^m$ be a  subgroup isomorphic to $\Z_n^t$. For a given generating set $B=\{x_1,\ldots ,x_t\}$ of $H$, we denote by
	$$
		F(x_1,\ldots,x_t)=\{x_i + \sum_{j=i+1}^t a_{j} x_j \; : \; i\in\{1,\ldots,t\} \textrm{ and each } a_{j}\in \Z_n\}.
	$$
	Let $Y=\{ y_1,\ldots ,y_m\}$ be a generating set of $\Z_n^m$. We say that  $X=\{x_1,\ldots ,x_t\}$  is an \emph{echelon} generating set  with respect to $Y$ if there are integers $1\le k_1\le \cdots \le k_t\le m$ such that
	$$
	x_i=y_{k_i}+ \sum_{k_i < j} \alpha_{i,j} y_j,\; i=1,\ldots ,t, \text{ for some } \alpha_{i,j}\in \Z_n.
	$$
	We note that, since $|Y|=m$ and $|X|=t$, all elements in $Y$ and in $X$ must have order $n$. When $Y=\{ e_1,\ldots ,e_m\}$  we omit the reference to $Y$ and simply say that $X$ is an echelon generating set.

	\begin{lemma}[Graham-Leeb-Rothschild for groups] \label{lem:mon_subg}
		For any positive integers $r, m$ and $n$,  there is  a positive integer $M=M(r,m,n)$ with the following property.  For every $r$--coloring of the $n$--cyclic subgroups $X \cong \mathbb{Z}_n$ of $\mathbb{Z}_n^{M}$,  there is an echelon generating set $\{ x_1,\ldots ,x_m\}$  such that the set of all $n$--cyclic subgroups generated by elements in $F(x_1,\ldots ,x_m)$  is monochromatic.
	\end{lemma}
	
	The proof of Lemma  \ref{lem:mon_subg} is an adaptation of the one by
Spencer~\cite{spen79} of the Vector Space Ramsey Theorem
which can be found in \cite{grahrotspen90}. We will use the same strategy and
notation. We note that, if $n$ is a prime, then
Lemma \ref{lem:mon_subg} can be derived from the Vector Space Ramsey Theorem.
The validity of the analog of this Theorem for abelian groups of the form
$\Z_n^M$ when $n$ is not a prime
is settled by Voigt \cite{voi80}. The application of this version to our
present needs, however, is not straightforward.
We therefore give a direct proof of the weaker version stated in Lemma 
\ref{lem:mon_subg}
(see Section~\ref{sec:conclusions} for further discussion on this question.)

We postpone the proof of   Lemma~\ref{lem:mon_subg} to Subsection~\ref{ss.prf_glr_groups} and proceed  to show Lemma~\ref{l.cycl_1_solution}.

	\subsection{Proof of Lemma~\ref{l.cycl_1_solution}}

	The first step in the proof is to show that one can always find solutions to the homogenous linear system $Ax=0$ within a set of the form $F(x_1,\ldots ,x_m)$.

	\begin{lemma}\label{lem.find_sol_in_struct}
		Let $A$ be a $k\times m$ integer matrix satisfying the $n$-column condition.
	Let $G\cong \Z_n^M$ and let $x_1,\ldots,x_m$ be $m$ elements in $G$ such that $\langle x_1,\ldots,x_m \rangle \cong \Z_n^m$.

	There are elements $y_1,\ldots,y_m$ in $F(x_1,\ldots,x_m)$ such that $y=(y_1,\ldots ,y_m)$ is a solution of the linear system $Ax=0$. In particular, each $y_i$ has order $n$.
	\end{lemma}

	\begin{proof}
	Since $A$ satisfies the $n$-columns condition, we may assume that the columns of $A$ are ordered in such a way that there are integers $1\leq k_1 < k_2 < \cdots < k_t =m$ such that
	\begin{description}
	\item[(i)] $S^1=\sum_{j=1}^{k_1} A^j=0 \in \Z_n^k$
	\item[(ii)] for $1<i\leq t$, $S^i=\sum_{j=k_{i-1}+1}^{k_i} A^j$ can be expressed as a linear combination of the columns $A^1,\ldots, A^{k_{i-1}}$ with coefficients in $\Z_n$.
	\end{description}

	Let $S^i=\sum_{j=1}^{k_{i-1}}\lambda_{i,j} A^j$, with $\lambda_{i,j}\in \Z_n$ be the linear combination of $S^i$ in terms of $A^j$'s.  Set $F=F(x_1,\ldots ,x_m)$. We construct inductively a solution $y\in F^m$ as follows.

	We set $y_i^1=x_1$ for $i\in [1,k_1]$. It follows from (i) that
	$$(A^1\ldots A^{k_1})(y_1^1,\ldots,y_{k_i}^1)^T=0.
	$$
	Assume that
	$$(A^1,\ldots,A^{k_i})(y_1^i,\ldots,y_{k_i}^i)^T=0,$$
	for some $1\le i<t$  and define
	 $$y_j^{i+1}=\left\{
	               \begin{array}{ll}
	                 y_j^i-\lambda_{i+1}x_{i+1}, & j\in [1,k_i] \\
	                 x_{i+1}, & j\in [k_i+1, k_{i+1}].
	               \end{array}
	             \right.
	$$
	Then,
	\begin{align*}
	(A^1,\ldots,A^{k_{i+1}})(y_1^{i+1},\ldots,y_{k_{i+1}}^{i+1})^T&=(A^1,\ldots,A^{k_{i+1}})(y_1^{i},\ldots,y_{k_{i}}^{i},0,\ldots ,0)^T\\
	&+(A^1,\ldots,A^{k_{i+1}})(-\lambda_{i+1,1},\ldots ,-\lambda_{i+1,k_i},1,\ldots ,1)x_{i+1},
	\end{align*}
	where the first summand is zero by induction and the second one because of the $n$--column property.

	Notice that, when we have finished with the recursion, we obtain an element of
	$F(x_1,\ldots,x_t)^m$, where $t$ is the number of classes of the partition of the columns.
	\end{proof}

	We are now ready to prove Lemma~\ref{l.cycl_1_solution}

	\begin{proof}[Proof of Lemma~\ref{l.cycl_1_solution}]
	Let $M=M(r,m,n)$ be the value for which the conclusion of Lemma~\ref{lem:mon_subg} holds. 
	Let $\chi:\Z_n^{M} \to [r]$ be an $r$--coloring of the elements of $\Z_n^M$. 

	Every element   $x=\sum_{i=1}x_ie_i\in\Z_n^M$ can be uniquely identified as the vector $(x_1,\ldots ,x_n)$ with $0\le x_i\le n-1$.  We define a liner ordering of the elements in $\Z_n^M$ by the lex order of its coordinates:    $(a_1,\ldots,a_M)<(b_1,\ldots,b_M)$ if and only if  $a_i=b_i$ for $i<j$ and $a_j<b_j$, where $j$ is the first entry in which the two vectors disagree.

	We define a colouring $\chi'$ on the $n$--cyclic groups isomorphic of $\Z_n^M$ as follows. For each $n$--cyclic subgroup $T<\Z_n^M$, 
	$$
	\chi'(T)=\chi(\min_{\le} \{y:\; \langle y \rangle =T\}),
	$$
	that is, $\chi'(T)$ is the color of its smallest generating element in the lex order.

	By Lemma~\ref{lem:mon_subg} there is an echelon generating set $\{ x_1,\ldots ,x_m\}$ of a subgroup $H<\Z_n^M$ isomorphic to $\Z_n^m$ such that all $n$--cyclic subgroups generated by elements in $F(x_1,\ldots ,x_m)$ have the same color.

	By Lemma~\ref{lem.find_sol_in_struct}, there are $y_1,\ldots ,y_m\in F(x_1,\ldots ,x_m)$ such that $y=(y_1,\ldots ,y_m)$ is a solution to $Ax=0$.

	The final observation is that, by the definition of $F(x_1,\ldots ,x_m)$ and the fact that $\{ x_1,\ldots ,x_m\}$ is an echelon generating set, every element in $F(x_1,\ldots ,x_m)$ has order $n$ and is the minimum element of the $n$--cyclic subgroup it generates. Indeed, let $z=x_i+\sum_{j=i+1}^m a_{j}x_j$ for some $i$ and some elements $a_{j}\in \Z_n$, be an element in $F(x_1,\ldots ,x_m)$. Since $\{ x_1,\ldots ,x_m\}$ is an echelon generating set, each element can be written as   $x_i=e_{k_i}+\sum_{k_{i}< j}\alpha_{i,j} e_j$, so that its leftmost nonzero coordinate is $1$, while the leftmost coordinates of $x_j, j>i$ have larger subscripts. Thus $z$ has the same order as $x_i$, which is $n$, and its leftmost nonzero coordinate is one, so that it is the smaller element in $\langle z\rangle$. 

	By the above observation, we have $\chi (y_i)=\chi'(\langle y_i\rangle)$ for each $i$ and the solution $y$ is monochromatic. This completes the proof. 
	\end{proof}


	\subsection{Proof of Lemma \ref{lem:mon_subg}} \label{ss.prf_glr_groups}

	The proof of Lemma~\ref{lem:mon_subg} is an adaptation of the one by Spencer~\cite{spen79} of the Vector Space Ramsey Theorem
which can be found in \cite{grahrotspen90}.
 We will use the same strategy and notation. 

	For a group $G\cong \Z_n^{M}$ and a positive integer $i\le n$, we denote by $[G]^i$ the family of cosets of subgroups of $G$ isomorphic to $\Z_n^i$. We call each such coset an $i$--{\it translate}. If $B\subset G$ is a $u$-translate, we also denote by $[B]^i$ the $i$-translates of $G$ in $B$.

	Let  $B\in [G]^{u+1}$ and let $p:B\to \Z_n^u$ be a surjective projection. We note that, for each $1$--translate $T\in [B]^1$, its projection $p(T)$ can be either a $1$--translate in $p(B)$, in which case we call $T$ \emph{transverse}, or a $0$--translate (a point), in which case we call $T$ \emph{vertical}, or none of the two, namely, $p(T)$ is a coset of a proper subgroup of $\Z_n$, in which case $T$ is said to be \emph{degenerate}. 

	We say that $B$ is \emph{special} relative to an $r$--coloring $\chi:[B]^1\to [r]$ and $p$ if for every pair of transverse $1$--translates $T,T'\in [B]^1$, 
	$$
	\mbox{if }\; p(T_1)=p(T_2)\in [p(B)]^1\; \mbox{ then }\; \chi (T_1)=\chi (T_2),
	$$
	that is, the color of a transverse $1$--translate in $B$ is defined by the color of its projection.

	The first step is to show that, for $M$ sufficiently large,  every $r$--coloring of $\Z_n^M$ has a special $(u+1)$--translate for the natural projection. In the proof we use the Hales--Jewett Theorem (see e.g. \cite[Theorem 2.2.3]{grahrotspen90}.) 

	\begin{lemma} \label{lem:tec_spec}
		For every positive integers $u$ and $r$, there exists $w=w(u,r)$ with the following property.
	Fix $p:\Z_n^{u+w} \to \Z_n^{u}$, the projection onto the first $u$ coordinates.
	For each colouring $\chi:[\Z_n^{u+w}]^1 \to [r]$, there is an echelon generating set $\{y_1,\ldots ,y_{u+1}\}$ and $b\in \Z_n^{u+w}$ such that the $(u+1)$--translate $B=b+\langle y_1,\ldots ,y_{u+1}\rangle$ is 
	special with respect to $p$ and $\chi$.
	\end{lemma}

	\begin{proof}
		Let ${\mathcal F}$ denote the family of functions $f:\Z_n^u\to \Z_n$ of the form
	$$
	f(x_1,\ldots,x_u)=c_0+c_1 x_1+\cdots + c_u x_u \mbox{ with } \; c_0,c_1,\ldots,c_u \in \mathbb{Z}_n.
	$$
	For every $F=(f_1,\ldots,f_w)\in {\mathcal F}^w$, where $w$ is to be specified later, we define the lifting $\overline{F}:\Z_n^u\to \Z_n^{u+w}$ by
	$$
	\overline{F}(x)=(x,F(x)).
	$$
	We observe that the image of a $1$--translate $T\in [\Z_n^u]^1$ by a lifting is a $1$--translate of $\Z_n^{u+w}$:  $\overline{F}(T)\in [\Z_n^{u+w}]^1$.
	Let $v$ be the number  of $1$--translates of $\Z_n^u$. Define a coloring $\chi'$ on ${\mathcal F}^w$ with $r^v$ colors by
	$$
	\chi'(F)=(\chi (\overline{F}(T)):\; T\in [\Z_n^u]^1).
	$$
	Set $w=HJ(|{\mathcal F}|, r^v)$, the constant for which the conclusion of the Hales--Jewett Theorem holds. By the Hales--Jewett Theorem, there is a   combinatorial line $L$ in ${\mathcal F}^w$ which is monochromatic by $\chi'$. By reordering coordinates, we may assume that
	\[ \label{eq:combin_line}
	L=\{(f,\ldots,f,f_{\eta+1},\ldots,f_w):\; f\in {\mathcal F}\},
	\]
	where $f_{\eta+1},\ldots,f_w$ are fixed. We set
	$$
	B=\bigcup_{F\in L} \overline{F}(\Z_n^u).
	$$
	Every element of $B$ is of the form $(x,F(x))$ where $x\in \Z_n^{u}$ and $F\in L$. Every element $F$ in the combinatorial line has, up to reordering of the coordinates,  $\eta$ equal components running over ${\cal F}$ and $w-\eta$ fixed components $f_{\eta+1},\ldots , f_{w}$.  Each $f_j$ is of the form $f_j(x_1,\ldots ,x_u)=c_0^j+c_1^jx_1+\cdots +c_u^jx_u$.
	Therefore,  with the elements
	\begin{align}
	y_i&=e_i+\sum_{j=u+\eta+1}^{u+w} c_i^j e_j , \; i=1,\ldots ,u\nonumber\\	
	y_{u+1}&=e_{u+1}+\cdots +e_{u+\eta}, \label{eq:yu+1}\\
	b&=\sum_{j=u+\eta+1}^{u+w} c_0^j e_j,\nonumber
	\end{align}
	where the $e_i$'s denote the standard generating set of $\Z_n^{u+w}$, we have $B=b+B_0$ with $B_0=\langle y_1,\ldots ,y_{u+1}\rangle$,  and $\{y_1,\ldots ,y_{u+1}\}$ is an echelon generating set.

	Let us show that $B$ is special. Let $T\in [B]^1$ be transverse, say $T=t_0+T_0$ for an $n$--cyclic subgroup $T_0<B_0$. Then $p(T)=p(t_0)+p(T_0)$ and, since $T$ is transverse, $p(T_0)$ is an $n$--cyclic subgroup of $\Z_n^u$. Hence, the first $u+\eta$ coordinates of an element $t\in T$ can be written as
	\begin{align}
		t_i&=x_i+(t_0)_i, \; i=1,\ldots,u \nonumber \\
		t_i&=(t_0)_i+[f(x)-f(0)]_{i-u},\; i=u+1,\dots, u+\eta \nonumber
	\end{align}
	for some $f\in {\mathcal F}$. Therefore, by taking $F=(f,\ldots ,f, f_{\eta+1}, \ldots ,f_{w})\in L$, we can write 
	$$
	T=\overline{F}(p(T)).
	$$
	Hence, if $T,T'\in [B]^1$ are two transverses with $p(T)=p(T')$ and $T=\overline{F}(p(T))$, $T'=\overline{F}'(p(T))$, the fact that $\chi' (F)=\chi' (F')$ (because both belong to the mono--$\chi'$ line $L$) implies, by the definition of $\chi'$, that $\chi (T)=\chi (T')$. Hence $B$ is special with respect to $\chi$ and $p$.
	\end{proof}

	We next prove the affine version of Lemma \ref{lem:mon_subg}. It uses the extended Hales--Jewett theorem for the existence of monochromatic $k$--dimensional combinatorial spaces in a coloring of the combinatorial cube (see e.g. \cite[Theorem~2.3.7]{grahrotspen90}.) In particular, when applied to a coloring of $\Z_n^M$ for sufficiently large $M$, we get a monochromatic $k$--coset of a subgroup isomorphic to $\Z_n^k$ which admits an echelon generating set.

	For brevity, we call the monochromatic structure we are looking for a $(t,n)$--{\it skeleton}. That is, a  $(t,n)$--skeleton is the set of $n$--cyclic subgroups generated by elements in $F(x_1,\ldots ,x_t)$ for some echelon generating set $\{ x_1,\ldots ,x_t\}$.

	\begin{lemma}[Affine version] \label{lem:aff_mon_subg}
	For any positive integers $r, t$ and $n$, there is  a positive integer $M=M(r,t,n)$ such that, for any $r$--coloring of the cosets $y+X$, $y \in \mathbb{Z}_n^{M}$, $X<\Z_n^M$ with $X \cong \mathbb{Z}_n$, there is a point $x_0\in \mathbb{Z}_n^{M}$ and a $(t,n)$-skeleton $S$, such that  $x_0+S$ is monochromatic.
	\end{lemma}
%

	\begin{proof} For  integers $t_1,\ldots ,t_r$, we denote by $N(t_1,\ldots ,t_r)$ the number such that, for $N\ge N(t_1,\ldots ,t_r)$, every $r$--coloring $\chi:[\Z_n^N]^1\to [r]$ of the $n$--cyclic subgroups of $\Z_n^N$ contains, for some $1\le i\le r$, a $(t_i,n)$--skeleton $S$ and $x_0\in \Z_n^N$ such that $x_0+S$ is monochromatic with color $i$. The Lemma follows by proving the existence of $N(t,\ldots ,t)$. 

	We use induction on the $r$--tuples $(t_1,\ldots ,t_r)$. Clearly, $N(1,\ldots,1)=1$ and $N(1,\ldots,1,0,1,\ldots,1)=0$. Suppose that $N(t_1,\ldots ,t_{i-1}, t_i-1,t_{i+1},\ldots ,t_r)$ exists for each $1\le i\le r$. We set
	\begin{align}
		s&= \max_{1\leq i \leq r}N(t_1,\ldots ,t_{i-1}, t_i-1,t_{i+1},\ldots ,t_r), \notag \\
		u&= EHJ(s,n,r), \notag \\
		w&= w(u,r), \notag \\
		N&= u+w,
	\end{align}
	where $EHJ(s,n,r)$ is the function for the extended Hales-Jewett Theorem and  $w(u,r)$ is the function from Lemma~\ref{lem:tec_spec}.

	Let  $\chi:[\Z_n^N]\to [r]$ be a given coloring of  the $1$-translates of $\Z_n^N$.
	By the definition of $w$ from Lemma~\ref{lem:tec_spec}, there is a
	 $(u+1)$-translate $B=b+B_0,$ and an echelon generating set of $B_0\cong  \Z_n^{u+1}$, that is special under the natural projection $p:B\to \Z_n^u$. We define a
	coloring $\chi'$ of the elements in $\Z_n^u$, by
	$$
	\chi'(x)=\chi(p^{-1}(x)),
	$$
	where $T_x=p^{-1}(x)$ is the unique vertical $1$-translate
	in $B$ that collapses completely onto $x$.

	By the definition of $u$, there exists an $s$-translate $X\subset \Z_n^u$
	monochromatic, say of color $1$, under $\chi'$. Then $p^{-1}(X)\subset B$, is an special $(s+1)$-translate all of whose vertical $1$-translates are colored $1$. We define a coloring $\chi''$ in
	$[X]^1$ by
	\begin{displaymath}
		\chi''(p(T))=\chi(T), \; \textrm{ for each transverse } \; T\in [B]^1.
	\end{displaymath}
	Since $p$ is special, this is a well defined coloring. By the induction hypothesis, as $s\leq N(t_{1}-1,t_2,\ldots,t_r)$, there exists a $(k,n)$--skeleton $S'\subseteq X$ and $x'_0\in \Z_n^N$ so that either
	\begin{itemize}
	\item[{\rm (i)}] $k=t_i$ for some $2\leq i\leq r$ and $x'_0+S'$ has color $i$ under $\chi''$, or
	\item[{\rm (ii)}] $k=t_{1}-1$ and $x'_0+S'$  has color $1$ under $\chi''$.
	\end{itemize}

	In case (i), $x_0'+S'$ is a $(t_i,n)$--skeleton with color $i$ and we are done.

	In case (ii), let  $S'=[F(y_1,\ldots, y_k)]^1$, where $y_1,\ldots ,y_k$ form an echelon generating set of a subgroup of $B_0$ isomorphic to $\Z_n^k$. By using the notation from the proof of Lemma~\ref{lem:tec_spec} on the structure of $B$, we can add $y_{k+1}:=y_{u+1}$ to these elements to form an echelon generating set $\{ y_1,\ldots ,y_k,y_{k+1}\}$  of a subgroup of $B_0$ isomorphic to $\Z_n^{k+1}$. Moreover, all the $n$--cyclic subgroups generated by elements $y_i+\sum_{j>i}a_jy_j\in F(y_1,\ldots ,y_{k+1})$ are, with respect to the projection $p$, either transverse (if $i<k+1$) or vertical (if $i=k+1$.) Hence, for $S= [F(y_1,\ldots ,y_{k+1})]^1$, we obtain the monochromatic $(t_1,n)$--skeleton $x_0+S$ with color $1$. This completes the proof.
	\end{proof}

	Lemma~\ref{lem:mon_subg} follows from Lemma~\ref{lem:aff_mon_subg} using a standard argument: color each $1$-translate of $\Z_n^{N}$ by the color of its corresponding $n$--cyclic subgroup.


	 \section{Counting Subgroups} \label{sec:counting}

	 In this section we give a counting result, Proposition \ref{p.num_subg_genrl} below, for the number of  subgroups isomorphic to  $\Z_n^M$ in an abelian group $G$ which will be used for the proof of the main result Theorem \ref{t.num_sol_bound_tor}, as well as the density result in Section \ref{s.dens}.  

	We shall use the following result by Yeh \cite{yeh48} regarding the number of subgroups of a $p$-group, $p$ a prime. A $p$-group $G$ is of type $(k_1,k_2,\ldots,k_\eta)$, $k_1\le\cdots\le k_{\eta}$,  if $G\cong \prod_{i=1}^\eta \Z_{p^{k_i}}$. 

	\begin{theorem}[Number of subgroups of a $p$-group, \cite{yeh48}]\label{t.num_subgroups_pg} 
		Let $G$ be a prime power abelian group of order $p^{k_1+k_2+\cdots +k_{\eta}}$, type $(k_1,k_2,\ldots,k_{\eta})$, where $k$'s are arranged in ascending order of magnitude. Let
	\begin{align} 
		 h_1&= h_2=\cdots = h_{m_1} > h_{m_1+1} = \cdots = h_{m_1+m_2} > \cdots  \label{e.one} \\ 
		&>  h_{m_1+m_2+\cdots+m_{r-1}+1}=\cdots=h_{m_1+m_2+\cdots + m_{r}}, \notag
	\end{align}
	where $m_1+m_2+\cdots + m_{r}=m\leq {\eta}$, be $m$ positive integers not greater than $k_{\eta}$, and let $\nu_i$ be such that $k_{\nu_i}<h_i\leq k_{\nu_i+1}$ ($i=1,2,\cdots,m$; $k_0=0$). Then the number of subgroups of type (\ref{e.one}) is given by
	\begin{align} \label{e.two}
	\left. p^h \prod_{i=1}^m (p^{{\eta}-\nu_i-i+1}-1) \middle/ \prod_{\mu=1}^r \prod_{\nu=1}^{m_{\mu}} (p^{\nu}-1) \right. \notag
	\end{align}
	where
	\begin{align*}
		h&= \sum_{i=1}^{m}({\eta}-\nu_i +1 - 2i)(h_i -1)  \\
		 &+ \frac{1}{2} (m_1^2+m_2^2 + \cdots + m_r^2-m^2) + \sum_{i=1}^m \sum_{\mu=0}^{\nu_i} k_{\mu}.
	\end{align*}
	\end{theorem}

	For our present purposes we denote by $\stiny{G}{M}_n$ the number of subgroups of $G$ isomorphic to $\Z_n^M$.  We next apply Theorem \ref{t.num_subgroups_pg} to prove Proposition~\ref{p.num_subg_genrl} below. We use the following simple Lemma, for which we omit its proof.

	\begin{lemma} \label{p.quo_cycl_max}
		Let $G=\Z_{n_1}\times \ldots \times \Z_{n_{s-1}}\times \Z_{n_s}$ be an abelian group with $n_1\vert \cdots \vert n_{s-1} \vert n_s$. If $H$ is a subgroup of $G$ isomorphic to $\Z_{n_s}$, then $$G/H\cong\Z_{n_1}\times \ldots \times \Z_{n_{s-1}}.$$
	\end{lemma}
	Let us notice that the statement of Lemma \ref{p.quo_cycl_max} is not true if $H$ is isomorphic to a smaller cyclic group.

	\begin{proposition}\label{p.num_subg_genrl}
		Let $G$ be an abelian group of exponent $n$ and let $M>1$ be an integer. If $\Z_{n}^M$ is a subgroup of $G$ then, 
	$$
			\sti{G}{M}_n \geq c_1|G|\sti{G/\Z_n}{M-1}_n,
	$$
	for some constant $c_1$ which depends only on $M$ and $n$.		
	\end{proposition}

	\begin{proof} Let $G=\prod_{p|n}G_p$ be the decomposition of $G$ into its $p$--components. The number of subgroups of $G$ isomorphic to $\Z_n^M$ is the product of the number of subgroups of $G_p$ isomorphic to $\Z_{p^{\beta_p}}^M$ for each prime $p$ dividing $n$, where $\beta_p$ is the largest power of $p$ dividing $n$:
	\begin{equation}\label{eq:pcomp}
	\sti{G}{M}_n=\prod_{p|n}\sti{G_p}{M}_{p^{\beta_p}}.
	\end{equation}

	Let $G_p$ be of type $(\alpha_1,\ldots , \alpha_l)$, where $\alpha_l=\beta_p$. Let $M'\ge M$ denote the number of copies of $\Z_{p^{\alpha_l}}$ in $G_p$, so that $\Z_{p^{\alpha_l}}^{M'}\subset G_p$ but $\Z_{p^{\alpha_l}}^{M'+1}\not\subset G_p$. 
	We can apply  Theorem~\ref{t.num_subgroups_pg} with $\eta=l$, $m=m_1=M$ and $h_1=\cdots=h_{m_1}=\alpha_l$. We then have 
	$\nu_i= l-M'$ for each $i=1,\ldots ,M$, and
	\begin{eqnarray}
	\sti{G_p}{M}_{p^{\alpha_l}}&=&p^h\prod_{i=1}^M\frac{p^{l-(l-M')-i+1}}{(p^{i}-1)}\nonumber  \\
	&\ge& p^{h+\sum_{i=1}^M(M'-2i+1)}.\label{eq:nsubg}
	\end{eqnarray}
	The value of $h$ given by Theorem~\ref{t.num_subgroups_pg} is, in our case,
	\begin{align*}
		h&= \sum_{i=1}^M (l-(l-M')+1-2i)(\alpha_l-1)  + \sum_{i=1}^M \sum_{\mu=0}^{\nu_i} k_{\mu}  \\
	&= \sum_{i=1}^M (M'-2i+1)(\alpha_l-1) + M\sum_{i=1}^{l-M'} \alpha_i.
	\end{align*}
	Hence the exponent of $p$ in \eqref{eq:nsubg} is
	$$
	\alpha_l\sum_{i=1}^M(M'-2i+1)+M\sum_{i=1}^{l-M'} \alpha_i=M\sum_{i=1}^l\alpha_i-M^2\alpha_l,
	$$
	which gives 
	$$
	\sti{G_p}{M}_{p^{\alpha_l}}\ge p^{-M^2\alpha_l}|G_p|^M.
	$$
	Therefore, in view of \eqref{eq:pcomp}, we have
	\begin{equation}\label{eq:lbound}
	\sti{G}{M}_{n}\ge c_1|G|^M,
	\end{equation}
	where $c_1=\prod_{p|n} p^{-M^2 \beta_p}$ depends only on $M$ and $n$.

	Let us compute an upper bound for the number of subgroups isomorphic to $\Z_{p^{\alpha_l}}^{M-1}$ in $G_p/\Z_{p^{\alpha_l}}$. We use Lemma~\ref{p.quo_cycl_max} to see that if $G_p$ is of  type
	\begin{displaymath}
		(\alpha_1,\alpha_2,\ldots,\alpha_{l-1},\alpha_l),
	\end{displaymath}
	 then $G_p/\Z_{p^{\alpha_l}}$ is of  type
	\begin{displaymath}
	(\alpha_1,\alpha_2,\ldots,\alpha_{l-1}).
	\end{displaymath}
	By using again Theorem~\ref{t.num_subgroups_pg} for the $p$--component $G_p$ of $G$, we have
	\begin{equation}\label{eq:ubound}
	\sti{G_p/\Z_{p^{\alpha_l}}}{M-1}_{p^{\alpha_l}}
		=p^h\prod_{i=1}^{M-1} \frac{(p^{(l-1)-(l-M')-i+1}-1)}{(p^{i}-1)} \le p^{h+(M-1)(M'-1)},
	\end{equation}
	where
	\begin{align*}
		h&=\sum_{i=1}^{M-1} ((l-1)-(l-M')+1-2i)(\alpha_l-1) + \sum_{i=1}^{M-1} \sum_{\mu=0}^{\nu_i} k_{\mu}\\
	&\leq(M-1)(M'-1)(\alpha_l-1) + (M-1)\sum_{i=1}^{l-M'} \alpha_i\\
	&= (M-1)(\sum_{i=1}^{l-1} \alpha_i-(M'-1)).
	\end{align*}
	By substituting this value of $h$ in \eqref{eq:ubound} we get
	$$
	\sti{G_p/\Z_{p^{\alpha_l}}}{M-1}_{p^{\alpha_l}}\le p^{(M-1)\sum_{i=1}^{l-1}\alpha_i}=|G_p|^{M-1}.
	$$
	By applying \eqref{eq:pcomp} to $G/\Z_{p^{\alpha_l}}$,
	\begin{equation}\label{eq:ubound1}
	\sti{G/\Z_{b}}{M-1}_n=\prod_{p|n}\sti{G_p/\Z_{p^{\beta_p}}}{M-1}_{p^{\beta_p}}\le |G|^{M-1}.
	\end{equation}
	By combining \eqref{eq:lbound} and \eqref{eq:ubound1} we get the result.	
	\end{proof}


	\section{Proof of Theorem \ref{t.num_sol_bound_tor}} \label{sec.cyclic_groups_all}

	We first introduce  the Removal Lemma for abelian groups. We recall that the
	$k$--determinantal of an integer matrix $A$ is the greatest common
	divisor of all the determinants of square submatrices of
	order $k$ of $A$.

	\begin{lemma}[Removal Lemma for Abelian Groups \cite{kraserven11-2}]\label{l.rm_eqsys_ab}
	Given an integer $(k\times m)$ matrix $A$ and  $\epsilon>0$ there is
	a $\delta=\delta (\epsilon, A)>0$ such that the following holds.

	For every Abelian group $G$ of order $n$ coprime with $d_k(A)$ and
	every family of subsets $X_1,\ldots ,X_m$ of $G$, if the homogeneous
	linear system $Ax=0$ has at most $\delta n^{m-k}$ solutions with
	$x_1\in X_1,\ldots ,x_m\in X_m$ then there are sets $X'_1\subset
	X_1,\ldots ,X'_m\subset X_m$, with $\max_i |X'_i|\le \epsilon n$, such
	that there are no solutions to the system with $x_1\in X_1\setminus
	X'_1,\ldots, x_m\in X_m\setminus X'_m$.
	\end{lemma}

	The next Proposition allows us to circumvent  the condition regarding the coprimality of
	$d_k(A)$ and $n$.

	\begin{proposition} \label{p.dettal_cond_rem}
		Let $A$ be a $k\times m$ integer matrix satisfying the $n$--columns condition.  Assume that $d_k(A)>1$.  There is   a $k\times m$  matrix $A'$ with integer coefficients which satisfies the $n$--columns condition such that $d_k(A')=1$  with the following property. For every abelian group $G$ with exponent $n$, the set of solutions of $A'x=0$ with  $x\in G$ is a subset of the set of solutions of the equation $Ax=0$ with  $x\in G$. 
	\end{proposition}

	\begin{proof} We proceed in two steps. First we note that a matrix $A$ satisfying the $n$--columns conditions is equivalent (in $\Z_n$) to a matrix $A''$ which satisfies the $\Z$--columns condition. Indeed, it suffices to replace the columns $A^i$ of $A$ by $A^i+nw_i$ for  appropriate integer vectors $w_i\in \Z^k$ so that the equations defining the $n$--columns condition are satisfied with the coefficients in $\Z$.  By doing so the set of solutions of $A''x=0$ in $G$ is the same as the one of the original linear system $Ax=0$, since $n$ is the exponent of $G$. 

	We thus may assume that $A$ satisfies the $\Z$--columns condition. 
	Let us consider the Smith Normal Form of $A$: there exist two matrices $U$ and $V$
	such that
	\begin{displaymath}
		U A V= (D|0)
	\end{displaymath}
	where $0$ is a $k\times (m-k)$ all-zero matrix,  $D$ is a $k\times k$ diagonal integer matrix with 
	$d_1,\ldots,d_k$ in the main diagonal with $\prod_{i=1}^k d_i=d_k(A)$. 
	Moreover, $U$ and $V$ are square unimodular integer matrices: 
	$U$ represents the row operations and $V$ the column operations that transform $A$ to $(D\vert 0)$. 

	Let us consider the matrix $B=(D|0) V^{-1}=U A$. As $B$ is built from $A$ by integer row operations, $B$ satisfies the $\Z$--columns condition and has the same set of solutions as $A$ in $G$. Since $V^{-1}$ is unimodular, it represents integer linear combinations of columns of $(D|0)$ and we can observe that each coefficient in the $i$-th row of $B$ is a multiple of $d_i$, the $i$-th element of $D$. 

	Consider now the matrix $A'$ obtained from $B$ by dividing the $i$-th row by $d_i$, for all $i\in[1,k]$. Since   $A' V=(\rm{Id}|0)$ is the Smith Normal Form of $A'$, we have  $d_k(B)=\det(\rm{Id})=1$. 

	Let $x=(x_1\ldots,x_m)\in G^m$ be a solution to $A'x=0$. By multiplying by $d_i$ the $i$--th linear equation we get
	\begin{displaymath}
	0=d_i(a'_{i,1}x_1 + \cdots + a'_{i,m}x_m)=b_{i,1}x_1 + \cdots + b_{i,m}x_m.
	\end{displaymath}
	Hence, $x$ is a solution of $Bx=0$ and therefore  it is also a solution to the homogeneous system defined by $A$.

	Finally let us show that $A'$ satisfies any linear equation satisfied by the columns of $B$ with the same coefficients (in particular, it satisfies the $\Z$--columns condition as $B$ does). Suppose that
	\begin{displaymath}
	\sum_{i=1}^m \lambda_{i}B^i=0,
	\end{displaymath} 
	with $\lambda_i\in \Z$.
	If we look at the $j$-th component, we observe that
	\begin{displaymath}
		0=\sum_{i=1}^m \lambda_{i}b_{j,i}=d_j \sum_{i=1}^m \lambda_{i}a_{j,i}',
	\end{displaymath}
	hence, we have 
	\begin{displaymath}
	\sum_{i=1}^m \lambda_{i}a_{j,i}'=0,
	\end{displaymath}
	that is, the same linear equation is satisfied by $A'$. 
	Finally, if $A'$ satisfies de $\Z$--columns condition, then in particular it satisfies the $n$--columns condition.
	\end{proof}

	We are now ready to prove Theorem~\ref{t.num_sol_bound_tor}.

	\begin{proof}[Proof of Theorem~\ref{t.num_sol_bound_tor}]
	Let $G$ be a sufficiently large group with exponent $n$. Let $\chi:G\setminus \{ 0\} \to [r]$ be an $r$--colouring of the nonzero elements of $G$. Let $A$ be a $k\times m$ integer matrix satisfying the $n$--columns condition. By Proposition~\ref{p.dettal_cond_rem} we may assume that $d_k(A)=1$. 

	For each divisor $d$ of $n$, denote by $M_d$ the positive integer given by Lemma~\ref{l.cycl_1_solution} such that every $r$--coloring of the nonzero elements $\Z_d^{M_d}$ has a monochromatic solution to the homogeneous linear system $Ax=0$ in $\Z_d^{M_d}\setminus \{ 0\}$ with every entry of $x$ with order $d$. Denote by
	$$
	M=\max_{d|n} M_d.
	$$
	For   fixed   $n$, every sufficiently large abelian group  with exponent $n$ contains a subgroup of the form $\Z_d^M$ for some divisor $d|n$. Let $n'$ be the largest divisor $d|n$ such that $G$ contains a subgroup isomorphic to $\Z_{d}^{M}$ and let $G'$ be the largest subgroup with exponent $n'$ in $G$. We observe  that
	$$
	|G'|\ge c_2|G|,
	$$
	where $c_2^{-1}=\prod_{d|n} d^M$, which depends only on $n$ and $M$. Indeed, if a divisor $d$ of $n$ does not divide $n'$ then, by the definition of $n'$, $G$ contains the product of at most $M$ copies of $\Z_d$. We also note that, as $n'|n$, the matrix  $A$ satisfies the $n'$--columns condition.

	Let $Y_1,\ldots,Y_r$ be the partition of $G'\setminus \{ 0\}$ defined by the coloring $\chi$.  By applying Lemma~\ref{l.rm_eqsys_ab} $r$ times, one for each color,  with $X_1=\cdots =X_m=Y_i$, for a given $\epsilon>0$ to be specified later there is $\delta =\delta (\epsilon, A)>0$ such that, if the system $Ax=0$ has less than $\delta |G'|^{m-k}$ monochromatic solutions, then there are  subsets $Y'_i\subset Y_i$ with $|Y'_{i}|\le \epsilon m |G'|$ such that the system $Ax=0$ has no solutions in $\cup_{i=1}^r (Y_i\setminus Y'_i)$. In particular there are at most $|\cup_{i=1}^r Y'_i|\le \epsilon rm|G'|$ removed elements with   order $n'$.

	Let $a\in \bigcup_{i=1}^r Y_i'$ be a removed element of order $n'$. The number of subgroups of $G'$ isomorphic to $\Z_{n'}^M$ which contain $a$ is the same, by Lemma~\ref{p.quo_cycl_max}, as the number $\stiny{G'/\Z_{n'}}{M-1}_{n'}$ of subgroups isomorphic to $\Z_{n'}^{M-1}$ in  $G'/\Z_{n'}$.  Therefore, by choosing
	$\epsilon$ in Lemma ~\ref{l.rm_eqsys_ab} such that
	\begin{equation}\label{eq:eps}
	0<\epsilon <\stiny{G'}{M}_{n'}/(rm |G'| \stiny{G'/\Z_{n'}}{M-1}_{n'}),
	\end{equation}
	there is a subgroup of $G'$ isomorphic to $\Z_{n'}^M$ with no element of order $n'$ removed. By Lemma~\ref{l.cycl_1_solution}, there is a monochromatic solution to $Ax=0$ in this subgroup, contradicting Lemma~\ref{l.rm_eqsys_ab}.  Thus there are at least $\delta |G'|^{m-k}\ge \delta c_2^{m-k}|G|^{m-k}$ monochromatic solutions. 

	We note that, by Lemma~\ref{l.cycl_1_solution}, 
	$$
	\frac{\stiny{G'}{M}_{n'}}{rm |G'| \stiny{G/\Z_{n'}}{M-1}_{n'}} \ge \frac{c_1}{rm},
	$$
	so that $\epsilon$ can be chosen independently of $|G'|$ for sufficiently large $|G'|$. 
	The Theorem follows by taking $c=\min\{\delta c_2^{m-k}, c_1 r^{-1}m^{-1}\}$. We note that $\delta$ depends only on $\epsilon$ and $A$, whereas $\epsilon$ depends on $c_1$, $r$ and $m$. The constants $c_1$ and $c_2$ depend on $n$ and $M$, while $M$, by Lemma~\ref{l.rm_eqsys_ab}, depends only on $r,n$ and $m$. We finally observe that the statement holds if we only consider matrices with coefficients in $\Z_n$. Since the number of $k\times m$ matrices with coefficients in $\Z_n$ is finite, the dependency on $A$ can be expressed as a dependency on $m$ and $n$ alone. 
	\end{proof}


	\section{Density version} \label{s.dens}

	In this section we characterize the $k\times m$ integer matrices such that, for every finite abelian group $G$, every set $S$ with positive density, $|S|>\epsilon |G|$ for some $\epsilon>0$, contains at least $\delta|G|^{m-k}$ solutions for some $\delta=\delta(\epsilon)>0$. This result is analogous to the version of Varnavides~\cite{var59} of the Szemer\'edi's theorem \cite{szem75} on arbitrarily long arithmetic progressions in dense sets of the integers. In this case every set of integers with positive asymptotic upper density contains  $cN^2$ $k$-arithmetic progressions for some positive constant $c$ which depends only on $k$.

	We say that a $(k\times m)$ matrix $A$ with integer coefficients and $m\ge
	k+2$ is {\it density regular} if, for every $\epsilon
	>0$ there is $n(\epsilon)\in \N$ such that the following holds: for
	every abelian group $G$ of order $n\ge n(\epsilon)$ and every subset
	$X\subset G$ such that $|X|\ge \epsilon n$, there is a nontrivial
	solution of the homogeneous linear system $Ax=0$ with all
	coordinates in $X$. Here by  trivial solution we mean one with all
	coordinates equal to the same common value.

	In the terminology of Rado's characterization of partition regular
	matrices, we say that the $k\times m$ integer matrix $A$, with $m\geq k+2$, verifies the {\it
	strong column} condition if the sum of the columns is the zero
	vector in $\Z^k$. Our main result is  the following:


	 \begin{theorem}[Counting for dense sets] \label{t.num_sol_dens}
		Let $A$ be a $k\times m$  integer matrix. Assume that $A$ satisfies the strong column condition. For every $\epsilon>0$, there exist a $\delta=\delta(\epsilon,A)>0$ with the following property:
	for every finite abelian group $G$ with large enough $|G|$ and for every set $X\subset G$ with $|X|\geq \epsilon |G|$, the linear system $Ax=0$ has $ \delta |G|^{m-k}$ solutions with $x\in X^m$.

	Moreover, if the matrix $A$ does not satisfy the strong column condition then $A$ is not density regular.
	\end{theorem}


	\begin{proof}[Proof of Theorem~\ref{t.num_sol_dens}]
	Assume $A$ satisfies the strong column condition. 
	Note that the strong column condition, $\sum_{i=1}^m A^i=0$, can be expressed as a linear combination with integer coefficients of the columns of $A$. Therefore, if $d_k(A)>1$, we use Proposition~\ref{p.dettal_cond_rem} to obtain a matrix $A'$ with $d_k(A')=1$, satisfying the strong column condition as well and whose solution set is a subset of the ones in $Ax=0$. We thus may assume that $d_k(A)=1$.

	Let $G$ be an abelian group. Let $\epsilon>0$ and let $X$ be a set with $|X|\ge \epsilon |G|$. Then we can find  trivial solutions to $Ax=0$, namely $x=(x_0,\ldots,x_0)$, for each $x_0\in X$.

	By the Removal Lemma for linear systems in  abelian groups, Lemma~\ref{l.rm_eqsys_ab},  there exists a $\delta=\delta(m^{-1}\epsilon/2,A)>0$ such that,  if there are less than $\delta |G|^{m-k}$ solutions to $Ax=0$, $x\in X^m$, then we can destroy all these solutions by removing at most $\epsilon/2 |G|$ elements from $X$. However, as we have not removed all the elements from $X$, there are, still, some trivial solutions. Therefore, the total number of solutions must be larger than $\delta |G|^{m-k}$.

	For the second part of the Theorem, suppose that there is one equation
	$a_1x_1+\cdots +a_mx_m=0$ with $\sum_ia_i=\alpha \neq 0$. Choose a
	sufficiently large positive integer $n$ and consider $G$ to be the cyclic
	group $\Z_n$. Let $X\subset \Z_n$ consist of the elements whose
	representatives in $[0,n]$ are congruent to $1$ modulo $|\alpha| +1$
	and lie in an initial segment $[0, n_0]$, where $n_0=n/(mt)$ and
	$t=\max_i|a_i|$. Thus $|X|\ge n/(mt (|\alpha|+1))$. Every element in $X^m$ is
	of the form $u'=u(|\alpha| +1)+\mathbf{1}$, where $u$ is an integer valued
	$m$--vector and $\mathbf{1}$ is the all ones vector. Hence, if
	$a=(a_1,\ldots ,a_m)$ and $u'\in X^m$, we have $$a\cdot
	u'=(a\cdot u)(|\alpha| +1)+\alpha.$$  Since $a\cdot u$ is an integer,
	$a\cdot u'$ cannot be equal to zero. Moreover, $u'$ is nonzero modulo $r$ because the elements in $X$ are in $[1,r_0]$, so $a\cdot u'\in [-r+1,r-1]$. Thus the equation $a_1x_1+\cdots
		+a_mx_m=0$ has no solutions in $X$.
	\end{proof}

\section{Final Remarks}\label{sec:conclusions}

We close this paper with some remarks on the hypothesis of our main result Theorem~\ref{t.num_sol_bound_tor}, and on the proof of the Ramsey result in Section \ref{sec:ramsey}.

 Our first remark concerns the hypothesis on the $n$--columns condition in the main result, Theorem~\ref{t.num_sol_bound_tor}.
In the original result of Rado on monochromatic solutions of linear systems on $\Z$, the columns condition is necessary and sufficient. The $p$--colorings which show the necessity of the condition cannot be translated to the case of finite abelian groups. 

When $n=p$ is a prime, then the $\Z_p$--columns condition is again necessary (Theorem~\ref{t.ff_1_solution}). The colorings which show the necessity, however, cannot be carried over to  the case of abelian groups with bounded exponent $n$ and $n$ non prime. More precisely, these colorings  only ensure that, if a matrix $A$ does not satisfies the $n$--column condition, then it has no monochromatic solutions in a set of the form $F(x_1,\ldots ,x_m)$ as described in Lemma~\ref{lem.find_sol_in_struct}. It is not clear to us if the $n$--columns condition is necessary for the conclusion of Theorem~\ref{t.num_sol_bound_tor} to hold.

On the other hand, there are weaker generalizations of the condition of \cite[Theorem~2.4]{berdeuhind92} that do not work. For example, just requiring that the given matrix $A$ satisfies the $p$--columns condition for every prime divisor $p$ of $n$ is not enough. For instance, the following matrix 
	\begin{equation*}
	A= \left(
	\begin{array}{cccc}
	1 & 0 & -1 & 0 \\
	0 & 1 & -1 & 0 \\
	0 & 0 & 0  & 2
	\end{array} \right).
\end{equation*}
satisfies  the $2$-columns condition, as the sum of the columns is $0$ with the coefficients in $\mathbb{Z}_2$. However, there is no solution with $x_4\in \mathbb{Z}_4^N\setminus \mathbb{Z}_2^N$, as that variable must satisfy $2\cdot x_4=0 \mod 4$ in all the $N$ coordinates and thus it has not order $4$.
If we use the $2$-coloring defined as: $\chi(x)=1$ if $x$ has order $4$ and $\chi(x)=2$ if $x$ has order $2$, there are only monochromatic solutions of order $2$. The number of solutions is $(2^N)^3$, yet is not as large as $\delta (4^N)^2$ that, for $N$ large enough, the corresponding version of 
Theorem~\ref{t.num_sol_bound_tor} would output.


Our second remark concerns the version of the Vector Space Ramsey Theorem we have used in our proof of the Ramsey result in Section~\ref{sec:ramsey}. In the terminology of Ramsey Theory, the Vector Space Ramsey Theorem is equivalent to say that the class of vector spaces is a Ramsey class. Even if Deuber and Rothschild~\cite{deurot78} show that the class of Finite Abelian Groups is not a Ramsey class,  Voigt \cite{voi80} characterizes the abelian groups $H$ for which the class of Finite Abelian Groups has the partition property with respect to $H$. This means that for any $G'$, there exist a $G$ such that, for every coloring of the subgroups of $G$ isomorphic to $H$, there is a subgroup isomorphic to $G'$ in $G$ all of its subgroups isomorphic to $H$ have the same color. In particular one can take $H$ to be a cyclic group, which is the case we are interested in. Bergelson, Deuber and Hindman \cite{berdeuhind92} give two proofs of Theorem \ref{t.ff_1_solution}, a combinatorial one and a second one based in ergodic theory. A close examination of their combinatorial proof convinced us that the simple application of the Vector Space Ramsey Theorem is not enough to reach the desired conclusion. Moreover, to adapt a complete argument to the context of finite groups seemed to be harder than to prove directly the existence of monochromatic skeletons. This prompted us to give a complete proof of Lemma~\ref{l.cycl_1_solution}, which can of course be also applied to the case of finite fields.

The third observation is connected with the restriction to the class of finite abelian groups with bounded torsion in the statement of Theorem~\ref{t.num_sol_bound_tor}. The counting result by Frankl, Graham and R\"odl \cite[Theorem~1]{frangrahrod88} provides an analogous statement for groups whose exponent  is linear with the order of the group. It is not clear to us if the two results can be combined in a  single statement for the class of all finite abelian groups. We note that
the constants involved in counting the number of solutions in both results depend heavily on the conditions on the exponent. Moreover, the proofs of the two statements are quite different, and each of them look for solutions with different structures. We believe that the conclusion of both results remains true for the whole class of finite abelian groups, but the combination of the two existent results may require new ideas.

\end{document}